\documentclass[11pt]{amsart}

\usepackage{amssymb}
\usepackage{bm}
\usepackage[centertags]{amsmath}
\usepackage{amsfonts}
\usepackage{amsthm}
\usepackage{mathrsfs}
\usepackage{mathabx}

\theoremstyle{definition}
\newtheorem{thm}{Theorem}[section]
\newtheorem{dfn}[thm]{Definition}
\newtheorem{lem}[thm]{Lemma}
\newtheorem{prp}[thm]{Proposition}
\newtheorem{cor}[thm]{Corollary}
\newtheorem{rmk}[thm]{Remark}

\newtheorem*{thm*}{Theorem}

\newtheorem*{cor*}{Corollary}

\newtheorem*{prp*}{Proposition}

\newtheorem*{rmk*}{Remark}

\newtheorem*{ntt}{Notation}

\newtheorem*{qst}{Question}

\newcommand{\inn}{\in\mathbb{N}}

\newcommand{\e}{\varepsilon}
\newcommand{\al}{\alpha}

\newcommand{\la}{\lambda}

\newcommand{\be}{\beta}

\newcommand{\si}{\sigma}
\newcommand{\cantor}{\{0,1\}^{\mathbb{N}}}
\newcommand{\ceq}{\sqsubseteq}
\newcommand{\cneq}{\varsqsubsetneq}

\newcommand{\varmin}{\widetilde{\min}}
\newcommand{\varmax}{\widetilde{\max}}

\newcommand{\cont}{2^{\aleph_0}}

\newcommand{\X}{\mathfrak{X}_{2^{\aleph_0}}}

\long\def\symbolfootnote[#1]#2{\begingroup%
\def\thefootnote{\fnsymbol{footnote}}\footnote[#1]{#2}\endgroup}

\begin{document}

\title[$\al$-Large Families]{$\al$-Large Families and Applications to Banach Space Theory}

\author[S. A. Argyros, P. Motakis]{Spiros A. Argyros and Pavlos
Motakis}
\address{National Technical University of Athens, Faculty of Applied Sciences,
Department of Mathematics, Zografou Campus, 157 80, Athens,
Greece} \email{sargyros@math.ntua.gr, pmotakis@central.ntua.gr}

\maketitle

\symbolfootnote[0]{\textit{2010 Mathematics Subject
Classification:} Primary 46B03, 46B06, 46B26, 03E05}

\symbolfootnote[0]{\textit{Key words:} Spreading models, Reflexive spaces, Non separable Banach spaces, Large families}

\symbolfootnote[0]{Research supported by API$\Sigma$TEIA program/1082.}

\begin{abstract}
The notion of $\al$-large families of finite subsets of an infinite set is defined for every countable ordinal number $\al$, extending the known notion of large families. The definition of the $\al$-large families is based on the transfinite hierarchy of the Schreier families $\mathcal{S}_\al, \al<\omega_1$. We prove the existence of such families on the cardinal number $2^{\aleph_0}$ and we study their properties. As an application, based on those families we construct a reflexive space $\X^\al$, $\al<\omega_1$ with density the continuum, such that every bounded non norm convergent sequence $\{x_k\}_k$ has a subsequence generating $\ell_1^\al$ as a spreading model.
\end{abstract}

\section*{Introduction}

One of the most significant examples of Banach spaces is Tsirelson space (see \cite{FJ}, \cite{T}), presented in the seventies. The main property of this space, is that it fails to contain a copy of $c_0$ or $\ell_p$, answering in the negative a problem posed by Banach. It is still an open problem whether there exist Tsirelson type spaces in the non-separable setting. A version of this problem has recently been solved in the negative direction in \cite{LT}, namely it is shown that spaces spanned by an uncountable basic sequence such that their norm satisfies an implicit formula, similar to the one of Tsirelson space (see \cite{FJ}), always contains a copy of $c_0$ or $\ell_p$. To be more precise, if $\kappa$ is an uncountable ordinal number, $\mathcal{B}$ is a hereditary and compact family of finite subsets of $\kappa$, $0<\theta<1$ is a real number, and $\|\cdot\|_{\theta,\mathcal{B}}$ is the unique norm defined on $c_{00}(\kappa)$ satisfying the following implicit formula

\begin{equation*}
\|x\|_{\theta,\mathcal{B}} = \max\big\{\|x\|_\infty,\;\sup\{\theta\sum_{i=1}^n\|E_ix\|_{\theta,\mathcal{B}}:\;\{E_i\}_{i=1}^d\;\text{is}\;\mathcal{B}-\text{admissible}\}\big\}
\end{equation*}
then the completion of $(c_{00}(\kappa),\|\cdot\|_{\theta,\mathcal{B}})$ contains a copy of $c_0$ or $\ell_p$.

As it seems not possible to have a non separable space, that strongly resembles Tsirelson space, a natural question is which properties of this space can be transferred to the non separable setting. Besides being reflexive, one of the main properties of Tsirelson space, is that it admits only $\ell_1$ as a spreading model, i.e. every bounded sequence without a norm convergent subsequence has a subsequence that generates a spreading model equivalent to the usual basis of $\ell_1$. The main goal of this paper is the construction of a non separable reflexive Banach space $\X$, with the aforementioned property.

\begin{thm*}
There exists a reflexive Banach space $\X$ generated by an unconditional basic sequence $\{e_\xi\}_{\xi<\cont}$, admitting only $\ell_1$ as a spreading model.
\end{thm*}

The construction of this space is based on the notion of $\al$-large families, which is defined as follows. If $A$ is an infinite set, $\mathcal{B}$ is a hereditary and compact family of finite subsets of $A$ and $\al$ is a countable ordinal number, we say that $\mathcal{B}$ is $\al$-large, if its restriction on every infinite subset of $A$, in a certain sense, contains a copy of $\mathcal{S}_\al$, the Schreier family of order $\al$. Equivalently, if its restriction on every infinite subset of $A$, has Cantor-Bendixson index, greater than or equal to $\omega^\al + 1$. We prove the existence of such families on the cardinal number $\cont$, by constructing for $\al<\omega_1$, $\mathcal{G}_\al$ an $\al$-large, hereditary and compact family of finite subsets of $\cantor$. We believe that these families  are of independent interest, as they retain some of the most important properties of the families $\mathcal{S}_\al, \al<\omega_1$. They are therefore a generalization of the Schreier families, defined on the continuum and a study of them is included in the paper.

In the first section of the paper, we define the notion of $\al$-large families of finite subsets of an infinite set and a brief study of them is given.

The second section is devoted to the construction of the families $\{\mathcal{G}_\al\}_{\al<\omega_1}$. Initially, using the Schreier family $S_1$ and diagonalization,  we recursively define some auxiliary families $\mathscr{G}_\al, \al<\omega_1$, which are subsets of $[\cantor]^{<\omega}\times\cantor$. The construction method used, imposes strong Schreier like properties on the families $\mathcal{G}_\al$, which are in fact the projection of $\mathscr{G}_\al$, on the component $[\cantor]^{<\omega}$. Next, properties of these families, which are crucial for the proof of the main result are included, among others, the fact that for $\al<\omega_1$, $\mathcal{G}_\al$ is an $\al$-large, compact and hereditary family of finite subsets of $\cantor$. Some additional results concerning the similarity of the $\mathcal{G}_\al$ to the $\mathcal{S}_\al$, $\al<\omega_1$ are proven.

The third section is concentrated on the construction of the space $\X$. The first step is the definition of a sequence of spaces $\{(X_n,\|\cdot\|_n)\}_n$, each one based on the family $\mathcal{G}_n$. In particular, the norm of these spaces is defined on $c_{00}(\cont)$ in a similar manner as the norm of Schreier space is defined on $c_{00}(\mathbb{N})$ (see \cite{Sch})  and they all have the unit vector basis $\{e_\xi\}_{\xi<\cont}$ as an unconditional Schauder basis. For $n\inn$, the main two properties of the space $X_n$ are the following. Firstly, every subsequence of the basis admits only $\ell_1$ as a spreading model and secondly the space $X_n$ is $c_0$ saturated. Next, using the spaces $X_n, n\inn$ and Tsirelson space $T$, a norm is defined on $c_{00}(\cont)$, in the following manner. For $x\in c_{00}(\cont)$, set
\begin{equation*}
\|x\| = \|\sum_{n=1}^\infty\frac{1}{2^n}\|x\|_ne_n\|_T.
\end{equation*}
The completion of $c_{00}(\cont)$ with this norm is the desired space $\X$, which has the unit vector basis $\{e_\xi\}_{\xi<\cont}$ as an unconditional Schauder basis. The proof of the fact that this space admits only $\ell_1$ as a spreading model, relies on the study of the behavior of the $\|\cdot\|_n$ norms on a normalized weakly null sequence $\{x_k\}_k$ in $\X$. Moreover, using the fact that the spaces $X_n$ are $c_0$ saturated, we prove that every subspace of $\X$ contains a copy of a subspace of $T$, which yields that the space is reflexive.

The fourth and final section concerns the construction, for $\al<\omega_1$, of reflexive spaces $\X^\al$ having an unconditional Schauder basis with size $\cont$, admitting $\ell_1^\al$ as a unique spreading model. The construction method used is a variation of the one used for the space $\X$.

\section{$\alpha$-large families}

We introduce the notion of $\al$-large families which concerns the complexity of a family $\mathcal{B}$ of finite subsets of a given infinite set $A$. This notion extends the well known concept of large families and it is defined using the transfinite hierarchy of the Schreier families $\{\mathcal{S}_\al\}_{\al<\omega_1}$, first introduced in \cite{AA}. After providing the definition of $\al$-large families we also give a useful characterization linking this notion with the Cantor-Bendixson index of a compact and hereditary family of finite subsets of a given infinite set.

\begin{ntt}
Let $A$ be a set, $\mathcal{B}$ be a family of subsets of $A$, $B$ be a subset of $A$ and $k$ be a natural number. We define
\begin{equation*}
[B]^k = \{F\subset B: \#F = k\}
\end{equation*}
and
\begin{equation*}
\mathcal{B}\upharpoonright B = \{F\in\mathcal{B}:\; F\subset B\}.
\end{equation*}

If $\mathcal{F}$ is a family of subsets of the natural numbers, $L$ is an infinite subset of $\mathbb{N}$ and $\phi:\mathbb{N}\rightarrow L$ is the uniquely defined order preserving bijection, we define
\begin{equation*}
\mathcal{F}[L] = \{\phi(F):\; F\in\mathcal{F}\}.
\end{equation*}
\end{ntt}

\begin{dfn}
Let $A$ be an infinite set and $\mathcal{B}$ a family of finite subsets of $A$.
\begin{itemize}

\item[(i)] We say that $\mathcal{B}$ is large, if for every $k\inn$, and $B$ infinite subset of $A$, we have that $[B]^k\cap\mathcal{B}\neq\varnothing$.

\item[(ii)] Given a countable ordinal number $\alpha$, we say that $\mathcal{B}$ is $\alpha$-large, if for every $B$ infinite subset of $A$, there exists a one to one map $\phi:\mathbb{N}\rightarrow B$, such that $\phi(F)\in\mathcal{B}$, for every $F\in\mathcal{S}_\alpha$.

\end{itemize}
\end{dfn}

\begin{rmk}
Using Ramsey theorem and a simple diagonalization argument, it is easy to see that $\mathcal{B}$ is large, if and only if it is 1-large.
\end{rmk}

The following lemma is an easy consequence of Theorem 1 from \cite{G}.

\begin{lem}
Let $\mathcal{F}, \mathcal{G}$ be hereditary and compact families of finite subsets of the natural numbers, such that for every $L$ infinite subset of the natural numbers, the Cantor-Bendixson index of $\mathcal{F}\upharpoonright L$, is strictly smaller than the Cantor-Bendixson index of $\mathcal{G}\upharpoonright L$. Then for every $M$ infinite subset of the natural numbers, there exists $L$ a further infinite subset of $M$, such that  $\mathcal{F}\upharpoonright L\subset \mathcal{G}\upharpoonright L$.\label{lemgas}
\end{lem}

\begin{prp}
Let $A$ be an infinite set, $\mathcal{B}$ be a hereditary and compact family of finite subsets of $A$ and $\al$ be a countable ordinal number. Then, the following assertions are equivalent.

\begin{itemize}

\item[(i)] $\mathcal{B}$ is $\al$-large.

\item[(ii)] For every $B$ infinite subset of $A$, the Cantor-Bendixson index of $\mathcal{B}\upharpoonright B$ is greater than or equal to $\omega^\al + 1$.

\end{itemize}
\label{alargechar}
\end{prp}

\begin{proof}
Given that (i) holds, (ii) is an immediate consequence of the fact that the Candor-Bendixson index of $\mathcal{S}_\al$ is equal to $\omega^\al + 1$ for every countable ordinal number $\al$ (see Proposition 4.10 from \cite{AA}).

 For the converse, we may clearly assume that $\mathcal{B}$ is a hereditary and compact family of finite subsets of the natural numbers. For a given countable ordinal $\al$, if (ii) holds, we shall prove the following statement.

  For every  infinite subset of the natural numbers $M$, there exists $L$ an infinite subset of $M$, such that $\mathcal{S}_\al[L] \subset \mathcal{B}$.

 The desired result evidently follows from the above. To prove this statement, we distinguish three cases.

  \noindent {\em Case 1:} $\al = 1$. Assume that for every  infinite subset of the natural numbers $M$, the Cantor-Bendixson index of $\mathcal{B}\upharpoonright M$ is infinite. This means that every such $M$ contains as subsets elements of $\mathcal{B}$, of unbounded cardinality. Since $\mathcal{B}$ is hereditary, we conclude that it is large and therefore it also is 1-large.

 \noindent {\em Case 2:} $\al$ is a limit ordinal number. Then there is $\{\be_k\}_k$ a strictly increasing sequence of ordinal numbers with $\sup_k\be_k = \al$, such that $\mathcal{S}_\al = \cup_k\{F\in\mathcal{S}_{\be_k}: \min F\geqslant k\}$.

 Using Lemma \ref{lemgas}, choose $L_1\supset\cdots\supset L_k\supset\cdots$ infinite subsets of $M$, such that $\mathcal{S}_{\be_k}\upharpoonright L_k\subset \mathcal{B}$, for all $k$.

 Choose $L = \{\ell_1<\cdots<\ell_k<\cdots\}$ an infinite subsets of $M$, with $\ell_m\in L_k$, for every $m\geqslant k$. It is not hard to check that $\mathcal{S}_\al[L]\subset \mathcal{B}$.

 \noindent {\em Case 3:} $\al$ is a successor ordinal number. If $\al = \be + 1$, then the following holds.

 For every $M$ infinite subset of the naturals and $n\inn$, there exists $L$ a further infinite subset of $M$, such that $(\mathcal{S}_\be\ast\mathcal{A}_n)\upharpoonright L\subset\mathcal{B}$, where
 \begin{equation*}
 \mathcal{S}_\be\ast\mathcal{A}_n = \{\cup_{i=1}^nF_i:\; F_i\in\mathcal{S}_\be,\; i=1,\ldots,n\}
 \end{equation*}

  The above statement follows form Lemma \ref{lemgas} and the fact that the Cantor-Bendixson index of $\mathcal{S}_\be\ast\mathcal{A}_n$ is equal to $\omega^\be n+1 < \omega^\al$.

 Therefore, given $M$ an infinite subset of the natural numbers, we may choose $L_1\supset\cdots\supset L_n\supset\ldots$ infinite subsets of $M$ such that $(\mathcal{S}_\be\ast\mathcal{A}_n)\upharpoonright L_n\subset\mathcal{B}$.

 Choose $L = \{\ell_1<\cdots<\ell_n<\cdots\}$ an infinite subsets of $M$, with $\ell_m\in L_n$, for every $m\geqslant n$. Once more, it is not hard to check that $\mathcal{S}_\al[L]\subset \mathcal{B}$.
\end{proof}

\section{A transfinite sequence of compact and hereditary families of finite subsets of $\cantor$}

In this section we define a transfinite sequence $\mathcal{G}_\al$, $\al<\omega_1$ of compact and hereditary families of finite subsets of $\cantor$ with each $\mathcal{G}_\al$ being $\al$-large for $\al<\omega_1$.
We shall first recursively define an auxiliary transfinite sequence $\{\mathscr{G}_\al\}_{\al<\omega_1}$ of subsets of $[\cantor]^{<\omega}\times\cantor$, which will then be used to define the $\mathcal{G}_\al$ for $\al<\omega_1$. We then prove the main properties of these families and we conclude this section by showing the $\mathcal{G}_\al$ have some similar properties to the Schreier families $\mathcal{S}_\al$.

\begin{ntt}
For $\si = \{\si(i)\}_{i=1}^\infty$ and $\tau = \{\tau(i)\}_{i=1}^\infty$ in $\cantor$, we define $\si\wedge\tau$ and $|\si\wedge\tau|$ as follows.
\begin{itemize}

\item[(i)] $\si\wedge\tau = \si$ and $|\si\wedge\tau| = \infty$, if $\si = \tau$.

\item[(ii)] $\si\wedge\tau = \varnothing$ and $|\si\wedge\tau| = 0$, if $\si(1) \neq \tau(1)$.

\item[(iii)] $\si\wedge\tau = \{\si(i)\}_{i=1}^\ell$ and $|\si\wedge\tau| = \ell$, if $\si \neq \tau, \si(1) = \tau(1)$ and $\ell = \min\{i\inn: \si(i+1) \neq \tau(i+1)\}$

\end{itemize}

For $s = \{s(i)\}_{i=1}^k$ and $t = \{t(i)\}_{i=1}^\ell$ finite sequences of 0's and 1's, we say that $s$ is an initial segment of $t$ and write $s\ceq t$, if $k\leqslant \ell$ and $s(i) = t(i)$ for $i=1,\ldots,k$. We say that $s$ is a proper initial segment of $t$ and write $s\cneq t$, if $s\ceq t$ and $s\neq t$.
\end{ntt}

\begin{dfn}
We define $\mathscr{G}_1$ to be all pairs $(F,\si)$, where $F = \{\tau_i\}_{i=1}^d\in[\cantor]^{<\omega}$, $d\inn$ and $\si\in\cantor$, such that the following are satisfied.

\begin{itemize}

\item[(i)] $\si \neq \tau_i$ for $i=1,\ldots,d$

\item[(ii)] $\si\wedge\tau_1\neq\varnothing$ and if $d>1$, then $\si\wedge\tau_1 \cneq \si\wedge\tau_2 \cneq \cdots \cneq \si\wedge\tau_d$

\item[(iii)] $d \leqslant |\si\wedge\tau_1|$

\end{itemize}
Define $\varmin(F,\si) = |\si\wedge\tau_1|$ and $\varmax(F,\si) = |\si\wedge\tau_d|$.\label{defg1}
\end{dfn}

Assume that $\al$ is a countable ordinal number, $\mathscr{G}_{\be}$ have been defined for $\be<\al$ and that for $(F,\si)\in\mathscr{G}_\be$, $\varmin(F,\si)$ and $\varmax(F,\si)$ have also been defined.

\begin{dfn}
Let $\be<\al$, $(F_i,\si_i)_{i=1}^d$, $d\inn$ be a finite sequence of elements of $\mathscr{G}_{\be}$ and $\si\in\cantor$.

We say that $(F_i,\si_i)_{i=1}^d$ {\em is a skipped branching of $\si$ in $\mathscr{G}_{\be}$}, if the following are satisfied.

\begin{itemize}

\item[(i)] The $F_i, i=1,\ldots,d$ are pairwise disjoint

\item[(ii)] $\si\neq\si_i$ for $i=1,\ldots,d$

\item[(iii)] $\si\wedge\si_1\neq\varnothing$ and if $d>1$, then $\si\wedge\si_1 \cneq \si\wedge\si_2 \cneq \cdots \cneq \si\wedge\si_d$

\item[(iv)] $|\si\wedge\si_i| < \varmin(F_i,\si_i)$ for $i=1,\ldots,d$

\item[(v)] $d\leqslant |\si\wedge\si_1|$

\end{itemize}\label{defskipped}
\end{dfn}

\begin{dfn}
Let $\be<\al$, $\si\in\cantor$ and $(F_i,\si)_{i=1}^d$, $d\inn$ be a finite sequence of elements of $\mathscr{G}_{\be}$.

We say that $(F_i,\si)_{i=1}^d$ {\em is an attached branching of $\si$ in $\mathscr{G}_{\be}$} if the following are satisfied.

\begin{itemize}

\item[(i)] The $F_i, i=1,\ldots,d$ are pairwise disjoint

\item[(ii)] If $d>1$, then $\varmax(F_i,\si) < \varmin(F_{i+1},\si)$, for $i=1,\ldots,d-1$

\item[(iii)] $d\leqslant \varmin(F_1,\si)$

\end{itemize}\label{defattached}

\end{dfn}

We are now ready to define $\mathscr{G}_\al$, distinguishing two cases.

\begin{dfn}
If $\al$ is a successor ordinal number with $\al = \be + 1$, we define $\mathscr{G}_\al$ to be all pairs $(F,\si)$, where $F\in[\cantor]^{<\omega}$ and $\si\in\cantor$, such that one of the following is satisfied.

\begin{itemize}

\item[(i)] $(F,\si)\in\mathscr{G}_{\be}$.

\item[(ii)] There is $(F_i,\si_i)_{i=1}^d$ a skipped branching of $\si$ in $\mathscr{G}_{\be}$ such that $F = \cup_{i=1}^d F_i$.

In this case we say that $(F,\si)$ {\em is skipped}. Moreover set $\varmin(F,\si) = |\si\wedge\si_1|$ and $\varmax(F,\si) = |\si\wedge\si_d|$.

\item[(iii)] There is $(F_i,\si)_{i=1}^d$ an attached branching of $\si$ in $\mathscr{G}_{\be}$ such that $F = \cup_{i=1}^d F_i$.

In this case we say that $(F,\si)$ {\em is attached}. Moreover set $\varmin(F,\si) = \varmin(F_1,\si)$ and $\varmax(F,\si) = \varmax(F_d,\si)$.

\end{itemize}

If $\al$ is a limit ordinal number, fix $\{\be_n\}_n$ a strictly increasing sequence of ordinal numbers with $\sup_n\be_n = \al$. We define
\begin{equation*}
\mathscr{G}_\al = \bigcup_{n=1}^\infty\big\{(F,\si)\in\mathscr{G}_{\be_n}:\; \varmin(F,\si)\geqslant n\big\}
\end{equation*}

\label{defgn}
\end{dfn}

\begin{rmk}
If $\al$ is a limit ordinal number, the sequence $\{\be_n\}_n$ may chosen in such a manner that both
\begin{equation*}
\mathscr{G}_\al = \bigcup_{n=1}^\infty\big\{(F,\si)\in\mathscr{G}_{\be_n}:\; \varmin(F,\si)\geqslant n\big\}
\end{equation*}
and
\begin{equation*}
\mathcal{S}_\al = \bigcup_{n=1}^\infty\big\{F\in\mathcal{S}_{\be_n}:\; \min F\geqslant n\big\}
\end{equation*}
From now on, we shall assume that this is the case.\label{remgs}
\end{rmk}

\begin{rmk}
Translating Definitions \ref{defg1}, \ref{defskipped}, \ref{defattached} and \ref{defgn} one obtains the following.
\begin{itemize}

\item[(i)] If $(F,\si)\in\mathscr{G}_1$, then $\#F\leqslant\varmin(F,\si)$.

\item[(ii)] If $(F,\si)\in\mathscr{G}_{\be+1}$ and $(F_i,\si_i)_{i=1}^d$ is a skipped branching of $\si$ in $\mathscr{G}_{\be}$ such that $F = \cup_{i=1}^dF_i$, then we have that $d \leqslant \varmin(F,\si)$.

\item[(iii)] If $(F,\si)\in\mathscr{G}_{\be+1}$ and $(F_i,\si)_{i=1}^d$ is an attached branching of $\si$ in $\mathscr{G}_{\be}$ such that $F = \cup_{i=1}^dF_i$, then we have that $d \leqslant \varmin(F,\si)$.

\end{itemize}\label{rmkdmin}
\end{rmk}

We now proceed to prove some key properties of the families $\mathscr{G}_\al$.

\begin{lem}
Let $\si, \si^\prime, \tau \in \cantor$, not all equal. The following are equivalent.
\begin{itemize}

\item[(i)]$\si\wedge\tau \cneq \si\wedge\si^\prime$

\item[(ii)]$\si\wedge\tau = \si^\prime\wedge\tau$.

\end{itemize}\label{littlelemma}
\end{lem}

\begin{proof}
Assume that (i) holds. We have that $\tau(j) = \si(j) = \si^\prime(j)$, for $j=1,\ldots,|\si\wedge\tau|$. Whereas, for $j = |\si\wedge\tau| + 1$, we have that $\tau(j) \neq \si(j) = \si^\prime(j)$. Therefore, $|\si^\prime\wedge\tau| = |\si\wedge\tau|$, which means that $\si\wedge\tau = \si^\prime\wedge\tau$.

The inverse is proved similarly.
\end{proof}

\begin{lem}
Let $\al$ be a countable ordinal number and $(F,\si)\in\mathscr{G}_\al$. Then there exist $\tau_m, \tau_M$ in $F$ such that the following are satisfied.

\begin{itemize}

\item[(i)] $\varmin(F,\si) = |\si\wedge\tau_m|$ and $\varmax(F,\si) = |\si\wedge\tau_M|$

\item[(ii)] For $\tau\in F$ we have that $\si\wedge\tau_m \ceq \si\wedge\tau \ceq \si\wedge\tau_M$

\end{itemize}
Moreover, if $\al$ is a successor ordinal number with $\al = \be+1$ the following hold.
\begin{itemize}

\item[(iii)] If $(F,\si)$ is skipped and $(F_i,\si_i)_{i=1}^d$ is a skipped branching of $\si$ in $\mathscr{G}_\be$ such that $F = \cup_{i=1}^dF_i$, then for $i=1,\ldots,d$ and $\tau\in F_i$, we have that $\si\wedge\si_i = \si\wedge\tau$.

\item[(iv)] If $(F,\si)$ is attached and $(F_i,\si)_{i=1}^d$ is an attached branching of $\si$ in $\mathscr{G}_\be$ such that $F = \cup_{i=1}^dF_i$, then for $1\leqslant i < j \leqslant d$ and $\tau_1\in F_i, \tau_2\in F_j$, we have that $\si\wedge\tau_1 \cneq \si\wedge\tau_2$.

\end{itemize}\label{lemminmaxbetween}
\end{lem}

\begin{proof}
We prove this lemma by transfinite induction. For $\al=1$ the desired result follows immediately from the definition of $\mathscr{G}_1$. Assume now that $\al$ is a countable ordinal number and that the statement holds for every $(F,\si)\in\mathscr{G}_\be$, for every $\be<\al$. If $\al$ is a limit ordinal number, then the result follows trivially from the inductive assumption and the definition of $\mathscr{G}_\al$. Assume therefore that $\al = \be + 1$ and let $(F,\si)\in\mathscr{G}_{\al}$.

We treat first the case when $(F,\si)$ is skipped. Let $(F_i,\si_i)_{i=1}^d$ be a skipped branching of $\si$ in $\mathscr{G}_\be$, such that $F = \cup_{i=1}^dF_i$.

We first prove part (iii), i.e. for $\tau\in F_i$, we have that $\si\wedge\si_i = \si\wedge\tau$, $i=1,\ldots,d$.

By the inductive assumption, there exist $\tau_m^i\in F_i$ such that $\varmin(F_i,\si_i) = |\si_1\wedge\tau_m^i|$ and for every $\tau\in F_i$ we have that $\si_i\wedge\tau_m^i\ceq\si_i\wedge\tau$.

Since, by definition, $|\si\wedge\si_i| < \varmin(F_i,\si_i) = |\si_i\wedge\tau_m^i|\leqslant |\si_i\wedge\tau|$, it follows that $\si\wedge\si_i \cneq \si_i\wedge\tau$ and by Lemma \ref{littlelemma} $\si\wedge\si_i = \si\wedge\tau$.

Choosing any $\tau_m\in F_1$ and $\tau_M\in F_d$, it is easy to see that (i) and (ii) are satisfied.

Assume now that $(F,\si)$ is attached. Let $(F_i,\si)_{i=1}^d$ be an attached branching of $\si$ in $\mathscr{G}_\be$, such that $F = \cup_{i=1}^dF_i$.

By the inductive assumption, there exist $\tau_m^i,\tau_M^i\in F_i$ such that $\varmin(F_i,\si) = |\si\wedge\tau_m^i|, \varmax(F_i,\si) = |\si\wedge\tau_M^i|$ and for every $\tau\in F_i$ we have that $\si\wedge\tau_m^i \ceq \si\wedge\tau \ceq \si\wedge\tau_M^i$.

We will show that for $1\leqslant i < j\leqslant d$, we have that $\si\wedge\tau_M^i \cneq \si\wedge\tau_m^j$. This proves both (iv) and that $\tau_m = \tau_m^1, \tau_M = \tau_M^d$ have the desired properties.

However, this follows immediately from the fact that $|\si\wedge\tau_M^i| = \varmax(F_i,\si) < \varmin(F_j,\si) = |\si\wedge\tau_m^j|$.

\end{proof}

The following result is an immediate consequence of Lemma \ref{lemminmaxbetween}.

\begin{cor}
Let $\al$ be a countable ordinal number and $(F,\si)\in\mathscr{G}_\al$. Then the following hold.
\begin{itemize}

\item[(i)] $\varmin(F,\si) = \min\{|\si\wedge\tau|: \tau\in F\}$

\item[(ii)] $\varmax(F,\si) = \max\{|\si\wedge\tau|: \tau\in F\}$

\end{itemize}\label{corminmax}
\end{cor}

\begin{cor}
Let $\al$ be a countable ordinal number and $(F,\si)\in\mathscr{G}_\al$, such that $\#F\geqslant 2$. Then
\begin{equation*}
\varmin(F,\si) \leqslant \min\{|\tau_1\wedge\tau_2|: \tau_1,\tau_2\in F\; \text{with}\; \tau_1\neq\tau_2\}
\end{equation*}
\label{corminbound}
\end{cor}

\begin{proof}
Let $\tau_1\neq\tau_2$ be in $F$. By Lemma \ref{lemminmaxbetween}, there exists $\tau_m\in F$, such that $\varmin(F,\si) = |\si\wedge\tau_m|$ and $\si\wedge\tau_m \ceq \si\wedge\tau_1$ as well as $\si\wedge\tau_m \ceq \si\wedge\tau_2$. It follows that $\si\wedge\tau_m \ceq \tau_1\wedge\tau_2$. We conclude that $\varmin(F,\si) \leqslant |\tau_1\wedge\tau_2|$.
\end{proof}

\begin{lem}

Let $\al$ be a countable ordinal number and $(F,\si)\in\mathscr{G}_\al$, such that $\#F\geqslant 2$. Then there exists $\si^\prime\in\cantor$, such that $(F,\si^\prime)\in\mathscr{G}_\al$ and
\begin{equation*}
\varmin(F,\si^\prime) = \min\{|\tau_1\wedge\tau_2|: \tau_1,\tau_2\in F\; \text{with}\; \tau_1\neq\tau_2\}
\end{equation*}\label{lemmineq}

\end{lem}

\begin{proof}
We prove this lemma by transfinite induction on $\al$. Assume that $\al = 1$, $(F,\si)\in\mathscr{G}_1$, such that $\#F\geqslant 2$ and $F = \{\tau_i\}_{i=1}^d, d\geqslant 2$ such that the assumptions of Definition \ref{defg1} are satisfied. Then $\si\wedge\tau_1\cneq\si\wedge\tau_2$ and by Lemma \ref{littlelemma} we have that $\si\wedge\tau_1 = \tau_1\wedge\tau_2$. We conclude that $\varmin(F,\si) = |\si\wedge\tau_1| = |\tau_1\wedge\tau_2|$. Corollary \ref{corminbound} yields that $\varmin(F,\si) = \min\{|\tau_1\wedge\tau_2|: \tau_1,\tau_2\in F\; \text{with}\; \tau_1\neq\tau_2\}$ and hence, the desired $\si^\prime$ is $\si$ itself.

Assume now that $\al$ is a countable ordinal number and that the conclusion holds for every $\be<\al$.

If $\al$ is a limit ordinal number, choose $\{\be_n\}_n$ a strictly increasing sequence of ordinal numbers with $\sup_n\be_n = \al$, such that the assumptions of Definition \ref{defgn} are satisfied. Let $(F,\si)\in\mathscr{G}_\al$ with $\#F\geqslant 2$. Then there is $n\inn$ such that $(F,\si)\in\mathscr{G}_{\be_n}$ and $\varmin(F,\si)\geqslant n$. Corollary \ref{corminbound} yields the following.
\begin{equation}
\min\{|\tau_1\wedge\tau_2|: \tau_1,\tau_2\in F\; \text{with}\; \tau_1\neq\tau_2\} \geqslant n \label{lemmineqeq1}
\end{equation}
By the inductive assumption, there exists $\si^\prime\in\cantor$ with $(F,\si^\prime)\in\mathscr{G}_{\be_n}$ and $\varmin(F,\si^\prime) = \min\{|\tau_1\wedge\tau_2|: \tau_1,\tau_2\in F\; \text{with}\; \tau_1\neq\tau_2\}$. By \eqref{lemmineqeq1} we have that $\varmin(F,\si^\prime)\in\mathscr{G}_\al$.

Assume now that $\al$ is a successor ordinal number with $\al = \be+1$ and let $(F,\si)\in\mathscr{G}_\al$ with $\#F\geqslant 2$. If $(F,\si)\in\mathscr{G}_\be$, then the inductive assumption yields the desired result. If this is not the case, then $(F,\si)$ is either skipped, or attached. If it is attached, then there is $(F_i,\si_i)_{i=1}^d$ an attached branching of $\si$, such that $F = \cup_{i=1}^dF_i$. If $d=1$, then $(F,\si_1)\in\mathscr{G}_\be$ and by the inductive assumption we are done. Otherwise, choose $\tau_1\in F_1, \tau_2\in F_2$. Lemma \ref{lemminmaxbetween} (iii) yields that $\si\wedge\tau_1 = \si\wedge\si_1 \cneq \si\wedge\si_2 = \si\wedge\tau_2$ and by Lemma \ref{littlelemma} we have that $\si\wedge\tau_1 = \tau_1\wedge\tau_2$. We conclude that $\varmin(F,\si) = |\si\wedge\si_1| = |\si\wedge\tau_1| = |\tau_1\wedge\tau_2|$ and therefore, applying Corollary \ref{corminbound} we have that $\si$ is the desired $\si^\prime$.

If on the other hand $(F,\si)$ is attached, using similar reasoning, Lemma \ref{lemminmaxbetween} (iv) and Corollary \ref{corminmax}, we conclude the desired result.

\end{proof}

\begin{cor}
Let $\{(F_k,\si_k)\}_k$ be a sequence in $\bigcup_{\be<\omega_1}\mathscr{G}_\be$ such that $\lim_k\varmin(F_k,\si_k) = \infty$. Then, if $F$ is an accumulation point of $\{F_k\}_k$, we have that $\#F \leqslant 1$.\label{corcomal}
\end{cor}

\begin{proof}

Let $F$ be an accumulation point of $\{F_k\}_k$, and assume that there are $\tau_1\neq\tau_2$ in $F$. Then there exists $L$ an infinite subset of the natural numbers, such that $\tau_1,\tau_2\in F_k$, for every $k\in L$. Corollary \ref{corminbound} yields that $|\tau_1\wedge\tau_2| \geqslant \varmin(F_k,\si_k)$, for all $k\in L$. We conclude that $|\tau_1\wedge\tau_2| = \infty$, i.e. $\tau_1 = \tau_2$, a contradiction.

\end{proof}

The following two lemmas will both be useful in the sequel.

\begin{lem}
Let $\al$ be a countable ordinal number and $(F,\si)\in\mathscr{G}_\al$. Let also $\si^\prime\in\cantor$, such that $\si^\prime\wedge\tau = \si\wedge\tau$ for all $\tau\in F$. Then the following hold.
\begin{itemize}

\item[(i)] $(F,\si^\prime)\in\mathscr{G}_\al$

\item[(ii)] $\varmin(F,\si^\prime) = \varmin(F,\si)$ and $\varmax(F,\si^\prime) = \varmax(F,\si)$

\end{itemize}\label{lemsiprime}
\end{lem}

\begin{proof}
We prove this lemma by transfinite induction. The case $\al = 1$ follows easily from the definition of $\mathscr{G}_1$. Assume now that the result holds for every $\be<\al$. The case where $\al$ is a limit ordinal number is trivial, assume therefore that $\al = \be + 1$ and let $(F,\si)\in\mathscr{G}_{\al}$, $\si^\prime\in\cantor$ such that the assumptions of the lemma are satisfied. Notice that it is enough to show that (i) is true, since part (ii) of the conclusion follows immediately from (i) and Corollary \ref{corminmax}.

We treat first the case when $(F,\si)$ is skipped, i.e. there exists $(F_i, \si_i)_{i=1}^d$ a skipped branching of $\si$ in $\mathscr{G}_\be$, with $F = \cup_{i=1}^dF_i$. To show that $(F,\si^\prime)\in\mathscr{G}_{\al}$, it suffices to show that $(F_i, \si_i)_{i=1}^d$ is a skipped branching of $\si^\prime$.

Notice that it is enough to show that $\si\wedge\si_i = \si^\prime\wedge\si_i$ for $i=1,\ldots,d$, which, by Lemma \ref{littlelemma}, is equivalent to $\si\wedge\si_i\cneq \si\wedge\si^\prime$ for $i=1,\ldots,d$.

Fix $1\leqslant i \leqslant d$ and chose $\tau\in F_i$. Lemma \ref{lemminmaxbetween} (iii) yields that $\si\wedge\si_i = \si\wedge\tau = \si^\prime\wedge\tau$. Once more, Lemma \ref{littlelemma} yields that $\si\wedge\si_i = \si\wedge\tau\cneq\si\wedge\si^\prime$.

Assume now that $(F,\si)$ is attached, i.e. there exists $(F_i, \si)_{i=1}^d$ an attached branching of $\si$ in $\mathscr{G}_\be$, with $F = \cup_{i=1}^dF_i$. Since, by the inductive assumption, the conclusion holds for the $(F_i,\si), i=1,\ldots,d,\; \si^\prime$, it is straightforward to check that $(F_i, \si^\prime)_{i=1}^d$ an attached branching of $\si^\prime$ in $\mathscr{G}_\be$ and therefore $(F,\si^\prime)\in\mathscr{G}_{\al}$.
\end{proof}

\begin{lem}
Let $(F,\si)\in\bigcup_{\be<\omega_1}\mathscr{G}_\be$ and $\si^\prime\in\cantor$ such that $\si\wedge\tau \cneq \si^\prime\wedge\tau$ for all $\tau\in F$. Then, if $\al = \min\{\be:\;(F,\si)\in\mathscr{G}_\be\}$, $\al$ is not a limit ordinal number and the following hold.
\begin{itemize}

\item[(i)] If $\al=1$, then $\#F = 1$.

\item[(ii)] If $\al = \be + 1$, then there exists $\si^{\prime\prime}\in\cantor$ with $(F,\si^{\prime\prime})\in\mathscr{G}_{\be}$.

\end{itemize}\label{lemnminusone}
\end{lem}

\begin{proof}
The fact that $\al$ is not a limit ordinal number follows trivially from Definition \ref{defgn}.
The case $\al=1$ is easy, we shall therefore only prove the case $\al = \be + 1$. Since $(F,\si)\notin\mathscr{G}_{\be}$, it is either skipped or attached.

Assume first that there is $(F_i,\si_i)_{i=1}^d$ a skipped branching of $\si$ in $\mathscr{G}_{\be}$ with $F = \cup_{i=1}^dF_i$. If $d = 1$, then $\si^{\prime\prime} = \si_1$ is evidently the desired element of $\cantor$. We will therefore prove that $d=1$. Towards a contradiction, assume that $d\geqslant 2$ and choose $\tau_1\in F_1, \tau_2\in F_2$.

Lemma \ref{lemminmaxbetween} (iii) yields that $\si\wedge\tau_1 = \si\wedge\si_1 \cneq \si\wedge\si_2 = \si\wedge\tau_2$. By the assumption, $\si\wedge\tau_1 \cneq \si^\prime\wedge\tau_1$ and using Lemma \ref{littlelemma} we conclude that $\si\wedge\tau_1 = \si\wedge\si^\prime$. Similarly, we conclude that $\si\wedge\tau_2 = \si\wedge\si^\prime$. We have shown that $\si\wedge\si^\prime \cneq \si\wedge\si^\prime$, which is absurd.

If $(F,\si)$ is attached, then using similar arguments and \ref{lemminmaxbetween} (iv), one can prove the desired result.
\end{proof}

\begin{prp}
Let $\al$ be a countable ordinal number, $(F,\si)\in\mathscr{G}_\al$ and $G$ be a non empty subset of $F$. Then $(G,\si)\in\mathscr{G}_\al$.\label{prphereditary}
\end{prp}

\begin{proof}
We proceed by transfinite induction. For $\al=1$ the result easily follows from the definition of $\mathscr{G}_1$. Assume that the statement is true for every $\be<\al$. The case when $\al$ is a limit ordinal number is an easy consequence of the inductive assumption and Corollary \ref{corminmax}. Assume therefore that $\al = \be + 1$ and let $(F,\si)$ be in $\mathscr{G}_{\al}$ and $G\subset F$.

Consider first the case, when $(F,\si)$ is skipped and $(F_i)_{i=1}^d$ be a skipped branching of $\si$ in $\mathscr{G}_\be$, such that $F = \cup_{i=1}^dF_i$.

Set $\{i_1<\cdots<i_p\} = \{i\in\{1,\ldots,d\}: G\cap F_i\neq\varnothing\}$ and $G_j = G\cap F_{i_j}$ for $j=1,\ldots,p$. By the inductive assumption, $(G_j,\si_{i_j})$ is in $\mathscr{G}_\be$ for $j=1,\ldots,p$ and, evidently, it is enough to show that $(G_j,\si_{i_j})_{j=1}^p$ is a skipped branching of $\si$.

Obviously, assumptions (i), (ii) and (iii) from Definition \ref{defskipped} are satisfied.

Corollary \ref{corminmax} yields that $\varmin(F_{i_j},\si_{i_j}) \leqslant \varmin(G_j,\si_{i_j})$ and hence (iv) is satisfied. Moreover $p\leqslant d \leqslant |\si\wedge\si_1| \leqslant |\si\wedge\si_{i_1}|$, which means that (v) is also satisfied.

If on the other hand $(F,\si)$ is attached, using similar reasoning and Corollary \ref{corminmax}, the desired result can be easily proven.

\end{proof}

We are now ready to define the families $\mathcal{G}_\al$, for $\al<\omega_1$ and prove their main properties.

\begin{dfn}
For a countable ordinal number $\al$ we define
\begin{equation*}
\mathcal{G}_\al = \{F\subset\cantor:\;\text{there exists}\;\si\in\cantor\;\text{with}\;(F,\si)\in\mathscr{G}_\al\}\cup\{\varnothing\}
\end{equation*}\label{defggn}
\end{dfn}

\begin{rmk}
It is clear that $\{\mathcal{G}_n\}_{n<\omega}$ is an increasing family of finite subsets of $\cantor$. Proposition \ref{prphereditary} also yields that $\mathcal{G}_\al$ is hereditary for all $\al<\omega_1$.\label{rmkincr}
\end{rmk}

\begin{prp}
Let $\al$ be a countable ordinal number. Then $\mathcal{G}_\al$ is $\al$-large. In particular, for every $B$ infinite subset of $\cantor$ there exists a one to one map $\phi:\mathbb{N}\rightarrow B$ with $\phi(F)\in\mathcal{G}_\al$ for every $F\in\mathcal{S}_\al$ and $\al<\omega_1$.\label{prpnlarge}
\end{prp}

\begin{proof}
Let $B$ be an infinite subset of $\cantor$. Choose $\{\tau_k\}_k$ pairwise disjoint elements of $B$ and $\si\in\cantor$, with $\lim_k\tau_k=\si$, such that $\si\wedge\tau_k \cneq \si\wedge\tau_{k+1}$ for all $k\inn$. Define $\phi:\mathbb{N}\rightarrow B$, with $\phi(k) = \tau_k$.

We shall inductively prove that for every $\al<\omega_1$ and $F\in\mathcal{S}_\al$, the following holds.

\begin{itemize}

\item[(i)] $(\phi(F),\si)\in\mathscr{G}_\al$

\item[(ii)] $\varmin(\phi(F),\si) = |\si\wedge\tau_{\min F}|$ and $\varmax(\phi(F),\si) = |\si\wedge\tau_{\max F}|$

\end{itemize}

The case $\al=1$ can be easily derived from the definition of $\mathscr{G}_1$. Assume now that $\al$ is a countable ordinal number and that the statement is true for every $F\in\mathcal{S}_\be$ and $\be<\al$.

We treat first the case when $\al$ is a limit ordinal number. Choose $\{\be_n\}_n$ a strictly increasing
sequence of ordinal numbers with $\sup_n\be_n = \al$, such that
\begin{equation*}
\mathscr{G}_\al = \bigcup_{n=1}^\infty\big\{(G,\si^\prime)\in\mathscr{G}_{\be_n}:\; \varmin(G,\si^\prime)\geqslant n\big\}
\end{equation*}
as well as
\begin{equation*}
\mathcal{S}_\al = \bigcup_{n=1}^\infty\big\{F\in\mathcal{S}_{\be_n}:\; \min F\geqslant n\big\}.
\end{equation*}
Then, if $F\in\mathcal{S}_\al$, there exists $n\inn$ with $F\in \mathcal{S}_{\be_n}$ and $\min F \geqslant n$. The inductive assumption yields that $(\phi(F),\si)\in\mathscr{G}_{\be_n}$ and $\varmin(\phi(F),\si) = |\si\wedge\tau_{\min F}|\geqslant \min F \geqslant n$. We conclude that $(\phi(F),\si)\in\mathscr{G}_{\al}$ and, of course, $\varmin(\phi(F),\si) = |\si\wedge\tau_{\min F}|$.

Assume now that $\al = \be+1$ and let $F\in\mathcal{S}_{\al}$. Then there exist $\min F \leqslant F_1 <\cdots < F_d$ in $\mathcal{S}_\be$ with $F = \cup_{i=1}^dF_i$.

The inductive assumption yields that $(\phi(F_i),\si)_{i=1}^d$ is an attached branching of $\si$ in $\mathscr{G}_\be$ and hence $(\phi(F),\si)\in\mathscr{G}_{\al}$.

Moreover, $\varmin(\phi(F),\si) = \varmin(\phi(F_1), \si) = |\si\wedge\tau_{\min F_1}| = |\si\wedge\tau_{\min F}|$. Similarly, we conclude that $\varmax(\phi(F),\si) = |\si\wedge\tau_{\max F}|$.

\end{proof}

\begin{rmk}
With a little more effort, it can be proven that for $\al<\omega_1$, $\mathcal{G}_\al$ is not $\al+1$-large. In particular, there does not exists a one to one map $\phi:\mathbb{N}\rightarrow\cantor$, such that $\phi(F)\in\mathcal{G}_\al$, for every $F\in\mathcal{S}_{\al+1}$. To be even more precise, for every $A$ infinite subset of $\cantor$, there exists $B$ a countable subset of $A$, such that the Cantor-Bendixson index of $\mathcal{G}_\al\upharpoonright B$ is equal to $\omega^\al + 1$ for all $\al<\omega_1$. Since we do not make use of this fact, we omit the proof.
\end{rmk}

The main result concerning the families $\mathcal{G}_\al$, $\al<\omega_1$ is the following.

\begin{thm}
Let $\al$ be a countable ordinal number. Then $\mathcal{G}_\al$ is an $\al$-large, hereditary and compact family of finite subsets of $\cantor$.\label{thmcompact}
\end{thm}

\begin{proof}

All we need to prove, is that $\mathcal{G}_\al$ is compact and we do so by transfinite induction. Let us first treat the case $\al = 1$ and assume $F$ is in the closure of $\mathcal{G}_1$.

If $F$ is finite, since $\mathcal{G}_1$ is hereditary, then $F\in\mathcal{G}_1$. It is therefore sufficient to show that $F$ cannot be infinite. Since $\mathcal{G}_1$ is hereditary, we may assume that $F$ is countable and let $\{\tau_i: i\inn\}$ be an enumeration of $F$.

We conclude, that setting $F_k = \{\tau_i: i=1,\ldots,k\}$, then $F_k\in\mathcal{G}_1$ and $\#F_k = k$. Choose $\{\si_k\}_k$ a sequence in $\cantor$ such that $(F_k,\si_k)\in\mathscr{G}_1$ for all $k$.

Remark \ref{rmkdmin} yields that $k\leqslant\varmin(F_k,\si_k)$ for all $k$. On the other hand, by Corollary \ref{corminbound} we have that $\varmin(F_k,\si_k) \leqslant |\tau_1\wedge\tau_2|$. We conclude that $k\leqslant |\tau_1\wedge\tau_2|$ for all $k\inn$, which is obviously not possible.

Assuming now that $\al$ is a countable ordinal number such that $\mathcal{G}_\be$ is compact for every $\be<\al$, we will show that the same is true for $\mathcal{G}_{\al}$.

We treat first the case in which $\al$ is a limit ordinal number. Fix $\{\be_n\}_n$ a strictly increasing sequence of ordinal numbers with $\sup_n\be_n = \al$ such that
\begin{equation*}
\mathscr{G}_\al = \bigcup_{n=1}^\infty\big\{(F,\si)\in\mathscr{G}_{\be_n}:\;\varmin(F,\si)\geqslant n\big\}
\end{equation*}

Let $F$ be in the closure of $\mathcal{G}_{\al}$. As previously, if $F$ is finite then it is in $\mathcal{G}_{\al}$ and it is therefore enough to show that $F$ cannot be infinite. Once more, we may assume that $F = \{\tau_i: i\inn\}$. Setting $F_k = \{\tau_1,\ldots,\tau_k\}$, we have that $F_k\in\mathcal{G}_\al$, therefore there exist $\{\si_k\}_k$, with $(F_k,\si_k)\in\mathscr{G}_\al$.

Using Corollary \ref{corminbound} we have that $\varmin(F_k,\si_k) \leqslant |\tau_1\wedge\tau_2| = d$. In other words, $(F_k,\si_k)\in\mathscr{G}_{\be_{n_k}}$, with $n_k\leqslant d$ for all $k$. Passing, if necessary, to a subsequence, we have that $(F_k,\si_k)\in\mathscr{G}_{\be_{n_0}}$, for all $k$. We conclude that $F\in\mathcal{G}_{\be_{n_0}}$, in other words $\mathcal{G}_{\be_{n_0}}$ is not compact, which is absurd.

Assume now that $\al = \be + 1$. Let $F$ be in the closure of $\mathcal{G}_{\al}$. As previously, it is enough to show that $F$ cannot be infinite. Once more, we may assume that $F = \{\tau_i: i\inn\}$.

Set $F_k = \{\tau_i: i=1,\ldots,k\}$, for all $k$. Then $F_k\in\mathcal{G}_{\al}$, i.e. there exists $\si_k$ such that $(F_k,\si_k)\in\mathscr{G}_{\al}$. Setting $d = |\tau_1\wedge\tau_2|$, Corollary \ref{corminbound}, yields the following.
\begin{equation}
\varmin(F_k,\si_k)\leqslant d\quad\text{for all}\;k\label{eq1prpcompact}
\end{equation}
By Definition \ref{defgn}, Remark \ref{rmkdmin} and \eqref{eq1prpcompact}, for every $k\inn$, there exist $\{F_j^k\}_{j=1}^{m_k}$ pairwise disjoint sets in $\mathcal{G}_\be$, with $F_k = \cup_{j=1}^{m_k}F_j^k$ and $m_k\leqslant d$. Passing to a subsequence, we may assume that $m_k = m$, for all $k$.

By the compactness of $\mathcal{G}_\be$, we may pass to a further subsequence and find $G_1, G_2,\ldots,G_m\in\mathcal{G}_\be$, such that $\lim_kF_j^k = G_j$, for $j=1,\ldots,m$.

We conclude that $F = \lim_kF_k = \lim_k(\cup_{j=1}^mF_j^k) = \cup_{j=1}^mG_j$. Since $\cup_{j=1}^mG_j$ is a finite set, this cannot be the case.

\end{proof}

Although the initial motivation behind the definition of the $\mathcal{G}_\al$ families was the construction of a non-separable reflexive space with $\ell_1$ as a unique spreading model, we believe that they are of independent interest, as they retain many of the properties of the families $\mathcal{S}_\al$. They are therefore a version of these families, defined on the Cantor set $\cantor$. We present a few more properties the $\mathcal{G}_\al$ have in common with the $\mathcal{S}_\al$.

\begin{lem}
Let $\al<\be$ be countable ordinal numbers. Then there exists $n\inn$ such that $\{(F,\si)\in\mathscr{G}_\al:\;\varmin(F,\si)\geqslant n\} \subset \mathscr{G}_\be$.\label{lemgeqn}
\end{lem}

\begin{proof}
Fix $\al$ a countable ordinal number. We prove this proposition by means of transfinite induction, starting with $\be = \al + 1$. In this case the result follows from the definition of $\mathscr{G}_\be$, for $n = 1$.

Assume now that $\be$ is a countable ordinal number with $\al<\be$, such that the statement holds for every $\al<\gamma<\be$. If $\be = \gamma + 1$, by the inductive assumption, there exists $n\inn$, such that $\{(F,\si)\in\mathscr{G}_\al:\;\varmin(F,\si)\geqslant n\} \subset \mathscr{G}_\gamma$. Evidently, we also have that $\{(F,\si)\in\mathscr{G}_\al:\;\varmin(F,\si)\geqslant n\} \subset \mathscr{G}_\be$.

If $\be$ is a limit ordinal number, fix $\{\be_k\}_k$ a strictly increasing sequence of ordinal numbers, such that $\be = \lim_k\be_k$ and
\begin{equation*}
\mathscr{G}_\be = \bigcup_k\big\{(F,\si)\in\mathscr{G}_{\be_k}:\;\varmin(F,\si)\geqslant k\big\}
\end{equation*}
Choose $k_0\inn$ with $\al < \be_{k_0}$. By the inductive assumption, there exists $m\inn$, such that $\{(F,\si)\in\mathscr{G}_\al:\;\varmin(F,\si)\geqslant m\} \subset \mathscr{G}_{\be_{k_0}}$. Setting $n = \max\{k_0,m\}$, we have the desired result.

\end{proof}

\begin{lem}
Let $\al<\be$ be countable ordinal numbers. Then there exists $n\inn\cup\{0\}$ such that $\mathscr{G}_\al \subset \mathscr{G}_{\be + n}$.\label{lemplusn}
\end{lem}

\begin{proof}
Fix $\be$ a countable ordinal number. We proceed by transfinite induction on $\al$. In the case $\al = 1$, it is easily checked that $\mathscr{G}_1\subset\mathscr{G}_\be$. Assume now that $\al$ is a countable ordinal with $\al<\be$, such that the statement holds for every $\gamma<\al$. If $\al = \gamma + 1$, then by the inductive assumption there exists $n\inn\cup\{0\}$ with $\mathscr{G}_\gamma \subset \mathscr{G}_{\be+n}$. We conclude that $\mathscr{G}_\al \subset \mathscr{G}_{\be + (n+1)}$. If $\al$ is a limit ordinal, fix $\{\al_k\}_k$ a strictly increasing sequence of ordinal numbers, such that $\al = \lim_k\al_k$ and
\begin{equation*}
\mathscr{G}_\al = \bigcup_k\big\{(F,\si)\in\mathscr{G}_{\al_k}:\;\varmin(F,\si)\geqslant k\big\}
\end{equation*}
Lemma \ref{lemgeqn} yields that there exists $m\inn$ with $\{(F,\si)\in\mathscr{G}_\al:\;\varmin(F,\si)\geqslant m\} \subset \mathscr{G}_\be$. The inductive assumption, yields that for $k=1,\ldots,m-1$, there exists $n_k\inn\cup\{0\}$ with $\mathscr{G}_{\al_k}\subset\mathscr{G}_{\be + n_k}$. Setting $n = \max\{m,n_1,\ldots,n_{m-1}\}$, it can be easily checked that $\mathscr{G}_\al\subset\mathscr{G}_{\be + n}$.

\end{proof}

\begin{prp}
Let $\al<\be$ be countable ordinal numbers. Then there exists $n\inn$ such that
\begin{equation*}
\big\{F\in\mathcal{G}_\al:\;\#F\geqslant 2\;\text{and}\;\min\{|\tau_1\wedge\tau_2|: \tau_1,\tau_2\in F, \tau_1\neq\tau_2\}\geqslant n\big\} \subset \mathcal{G}_\be.
\end{equation*}
\label{prpgeqn}
\end{prp}

\begin{proof}
Let $\al<\be$ be countable ordinal numbers. Choose $n\inn$ such that the conclusion of Lemma \ref{lemgeqn} is satisfied. We show that this $n$ is the desired natural number. Let $F\in\mathcal{G}_\al$ with $\#F\geqslant 2\;\text{and}\;\min\{|\tau_1\wedge\tau_2|: \tau_1,\tau_2\in F, \tau_1\neq\tau_2\}\geqslant n$. Then there exists $\si\in\cantor$ with $(F,\si)\in\mathscr{G}_\al$. Lemma \ref{lemmineq} yields that there exists $\si^\prime\in\cantor$ such that $(F,\si^\prime)\in\mathscr{G}_\al$ and $\varmin(F,\si^\prime)\geqslant n$. By the choice of $n$, we have that $(F,\si^\prime)\in\mathscr{G}_\be$, i.e. $F\in\mathcal{G}_\be$.
\end{proof}

The following proposition is an obvious conclusion of Lemma \ref{lemplusn}

\begin{prp}
Let $\al<\be$ be countable ordinal numbers. Then there exists $n\inn\cup\{0\}$ such that $\mathcal{G}_\al \subset \mathcal{G}_{\be + n}$.\label{prpplusn}
\end{prp}

The following fact is proven in \cite{LT}, Proposition 7.4.
If $\kappa$ is an infinite cardinal number, then there exists a large, hereditary and compact family of finite subsets of $\kappa$, if and only if $\kappa$ is not $\omega$-Erd\H os. The following question arises naturally.

\begin{qst}
Let $\kappa$ be an infinite cardinal number which is not $\omega$-Erd\H os and $\alpha$ be a countable ordinal number. Does there exist an $\alpha$-large, hereditary and compact family of finite subsets of $\kappa$?
\end{qst}

\section{The space $\X$}

In this section we define the space $\X$ and prove that it is reflexive, has an unconditional Schauder basis of length the continuum and that it admits only $\ell_1$ as a spreading model. In the beginning we define a sequence of non separable spaces $X_n$, $n\inn$. Each one is defined using the family $\mathcal{G}_n$ in a similar manner as the Schreier family $\mathcal{S}_1$ is used to define the space in \cite{Sch}. Then the construction of $\X$ is presented, which combines the spaces $X_n$ and Tsirelson space, using a method first appeared in \cite{DFJP}. In the end the properties of the space $\X$ are deduced by directly using the structure of the families $\mathcal{G}_n$.

Before proceeding to the definition of the spaces $X_n$ and $\X$, let us first recall the notion of $\ell_1^\al$ spreading models.

\begin{dfn}
Let $\{x_k\}_k$ be a sequence in a Banach space and $\al$ be a countable ordinal number. We say that $\{x_k\}_k$ generates an $\ell_1^\al$ spreading model, if there exists a constant $c>0$ such that for every $F\in\mathcal{S}_\al$ and every real numbers $\{\lambda_k\}_{k\in F}$ the following holds:
\begin{equation*}
\|\sum_{k\in F}\lambda_kx_k\| \geqslant c\sum_{k\in F}|\lambda_k|.
\end{equation*}
\end{dfn}

Let us from now on fix a one to one and onto map $\tau \rightarrow \xi_\tau$ from $\cantor$ to the cardinal number $\cont$.

\begin{dfn}
For $n\inn$ define a norm on $c_{00}(\cont)$ in the following manner.
\begin{itemize}

\item[(i)]
For $n\inn$, we may identify an $F\in\mathcal{G}_n$ with a linear functional $F:c_{00}(\cont)\rightarrow\mathbb{R}$ in the following manner. For $x = \sum_{\xi<\cont}\la_\xi e_\xi \in c_{00}(\cont)$
\begin{equation*}
F(x) = \sum_{\tau\in F}\la_{\xi_\tau}
\end{equation*}

\item[(ii)]
For $x\in c_{00}(\cont)$ define
\begin{equation*}
\|x\|_n = \sup\{|F(x)|: F\in\mathcal{G}_n\}
\end{equation*}
Set $X_n$ to be the completion of $(c_{00}(\cont), \|\cdot\|_n)$.
\end{itemize}\label{defxn}

\end{dfn}

\begin{prp}
Let $n\inn$. Then the following hold.
\begin{itemize}

\item[(i)] The space $X_n$ is $c_0$ saturated.

\item[(ii)] The unit vector basis $\{e_\xi\}_{\xi<\cont}$ is a normalized, suppression unconditional and weakly null basis of $X_n$, with the length of the continuum.

\item[(iii)] Any subsequence of the unit vector basis admits only $\ell_1$ as a spreading model.

\end{itemize}\label{prpc0sat}

\end{prp}

\begin{proof}
To prove (i), notice that since $\mathcal{G}_n$ is compact and contains only finite sets, it is scattered. The main Theorem from \cite{PZ} yields that $C(\mathcal{G}_n)$ is $c_0$ saturated. Evidently, the map $T:X_n\rightarrow C(\mathcal{G}_n)$ with $Tx(F) = F(x)$ is an isometric embedding and therefore, $X_n$ is $c_0$ saturated.

Property (ii) follows from the fact that $\mathcal{G}_n$ is hereditary and property (iii) is a consequence of the fact that $\mathcal{G}_n$ is 1-large.
\end{proof}

\begin{rmk}
For a cardinal number $\kappa$ and $\mathcal{B}$ a compact, hereditary and large family of finite subsets of $\kappa$, one may define a $c_0$ saturated space $X_\mathcal{B}$, in the same manner as in Definition \ref{defxn}. Then any subsequence of the unit vector basis $\{e_\xi\}_{\xi<\kappa}$ admits only $\ell_1$ as a spreading model. Therefore the problem of finding a basic sequence of length $\kappa$, admitting only $\ell_1$ as a spreading model, is reduced to the existence of such a family $\mathcal{B}$. As it is proven in \cite{LT}, Proposition 7.4, this is equivalent to $\kappa$ not being $\omega$-Erd\H os.

It is also worth noting, that for a given cardinal number $\kappa$, it is easy to construct a reflexive space $X$ with a basis $\{e_\xi\}_{\xi<\kappa}$, having the property that every subspace has a sequence admitting $\ell_1$ as a spreading model. As proven in \cite{D}, \cite{L} and \cite{S}, any space $X$ with an unconditional basis, embeds as a complemented subspace in a space with a symmetric basis $D$. As noted in \cite{AM}, the construction in \cite{D} has the following additional property. Every subspace of $D$ contains a copy of a subspace of $X$. One may therefore embed Tsirelson space $T$ into a space $D$ with a symmetric basis $\{e_n\}_n$, saturated with subspaces of $T$. Since this basis is symmetric, it may naturally be extended to a basis $\{e_\xi\}_{\xi<\kappa}$ to define a space $X$ having the desired property. However, this space also admits spreading models not equivalent to $\ell_1$. For instance, the basis itself being symmetric and not equivalent to $\ell_1$, fails this property.

\end{rmk}

By $T$ we denote Tsirelson space as defined in \cite{FJ} and by $\{e_n\}_n$ we denote its usual basis. We are now ready to define the space $\X$, using the spaces $X_n$, Tsirelson space $T$ and a method first appeared in \cite{DFJP}.

\begin{dfn}
Define the following norm on $c_{00}(\cont)$. For $x\in c_{00}(\cont)$
\begin{equation*}
\|x\| = \big\|\sum_{n=1}^\infty\frac{1}{2^n}\|x\|_ne_n\big\|_T
\end{equation*}
Set $\X$ to be the completion of $(c_{00}(\cont), \|\cdot\|)$.\label{defx}
\end{dfn}

Set $\la = \|\sum_{n=1}^\infty \frac{1}{2^n}e_n\|_T$ and for $\xi<\cont$, $\tilde{e}_\xi = \frac{1}{\la}e_\xi$. Since $\{e_\xi\}_{\xi<\cont}$ is normalized and suppression unconditional in $X_n$, and $\{e_n\}_n$ is 1-unconditional in T, we conclude that $\{\tilde{e}_\xi\}_{\xi<\cont}$ is a normalized suppression unconditional basis of $\X$.

For $n\inn$ define $P_n:\X\rightarrow X_n$ with $P_nx = \frac{1}{2^n}x$. Evidently $P_n$ is well defined and $\|P_n\|\leqslant 1$, for all $n\inn$.\vskip5pt

The main result is the following, which is a combination of Proposition \ref{prpell1smwn} and Corollary \ref{correflexive}, which will be presented in the sequel.

\begin{thm}
The space $\X$ is a non separable reflexive space with a suppression unconditional Schauder basis with the length of the continuum, having the following property. Every normalized weakly null sequence in $\X$ has a subsequence that generates an $\ell_1^n$ spreading model, for every $n\inn$.\label{mainthm}
\end{thm}

\begin{lem}
Let $\{\tilde{e}_{\xi_k}\}_k$ be a subsequence of the basis $\{\tilde{e}_\xi\}_{\xi<2^{\aleph_0}}$ of $\X$. Then it has a subsequence that generates an $\ell_1^n$ spreading model for every $n\inn$.\label{lembasisell1}
\end{lem}

\begin{proof}
Set $B = \{\tau: \xi_\tau = \xi_k$ for some $k\inn\}$. By Proposition \ref{prpnlarge} there exists a one to one map $\phi:\mathbb{N}\rightarrow B$ such that $\phi(F)\in\mathcal{G}_n$ for every $F\in\mathcal{S}_n$ and $n\inn$.

Pass to $L$ an infinite subset of the natural numbers such that the map $\tilde{\phi}:L\rightarrow 2^{\aleph_0}$ with $\tilde{\phi}(j) = \xi_{\phi(j)}$ is strictly increasing. We will show that $\{\tilde{e}_{\xi_{\phi(j)}}\}_{j\in L}$ admits an $\ell_1^n$ spreading model for every $n\inn$.

By unconditionality, it is enough to show that there are positive constants $c_n$ such that for every $n\inn$, $F\in\mathcal{S}_n$, $F\subset L$ and $\{t_j\}_{j\in F}$ positive real numbers, we have that
\begin{equation*}
\|\sum_{j\in F}t_j\tilde{e}_{\xi_{\phi(j)}}\| \geqslant c_n\sum_{j\in F}t_j
\end{equation*}
By definition, we have that $\|\sum_{j\in F}t_j\tilde{e}_{\xi_{\phi(j)}}\| \geqslant \frac{\la}{2^n}\|\sum_{j\in F}t_j e_{\xi_{\phi(j)}}\|_n$ and by the choice of $\phi$, we have that $\phi(F)\in\mathcal{G}_n$. Hence, $\phi(F)(\sum_{j\in F}t_j e_{\xi_{\phi(j)}}) = \sum_{j\in F}t_j$ which yields that $\|\sum_{j\in F}t_j e_{\xi_{\phi(j)}}\|_n = \sum_{j\in F}t_j$.

We finally conclude that $\|\sum_{j\in F}t_j\tilde{e}_{\xi_{\phi(j)}}\| \geqslant \frac{\la}{2^n}\sum_{j\in F}t_j$
\end{proof}

\begin{prp}
Let $\{x_k\}_k$ be a normalized, disjointly supported block sequence of $\{\tilde{e}_\xi\}_{\xi<\cont}$, such that $\lim\sup_k\|x_k\|_\infty > 0$. Then $\{x_k\}_k$ has a subsequence that generates an $\ell_1^n$ spreading model for every $n\inn$.\label{supnorm}
\end{prp}

\begin{proof}
By unconditionality, it is quite clear, that by passing, if necessary, to a subsequence of $\{x_k\}_k$, there exist $\e>0$ and $\{\tilde{e}_{\xi_k}\}_k$ a subsequence of $\{\tilde{e}_\xi\}_{\xi<\cont}$, such that for any $\la_1,\ldots,\la_m$ real numbers, one has that
\begin{equation*}
\|\sum_{k=1}^m\la_kx_k\| > \e\|\sum_{k=1}^m\la_k\tilde{e}_{\xi_k}\|
\end{equation*}
Lemma \ref{lembasisell1} yields the desired result.
\end{proof}

\begin{prp}
Let $\{x_k\}_k$ be a normalized block sequence in $\X$, such that $\lim_k\|P_nx_k\|_n = 0$, for all $n\inn$. Then $\{x_k\}_k$ has a subsequence equivalent to a block sequence in $T$. In particular, $\{x_k\}_k$ has a subsequence that generates an $\ell_1^n$ spreading model for every $n\inn$.\label{pnzero}
\end{prp}

\begin{proof}
Using a sliding hump argument, it is easy to see, that passing, if necessary, to a subsequence of $\{x_k\}_k$, there exist $\{I_k\}_k$ increasing intervals of the natural numbers, such that if we set $y_k = \sum_{n\in I_k}\frac{1}{2^n}\|x_k\|_ne_n$, then $\{x_k\}_k$ is equivalent to $\{y_k\}_k$.
\end{proof}

\begin{lem}
Let $\{x_k\}_k$ be a normalized, disjointly supported block sequence of $\{\tilde{e}_\xi\}_{\xi<\cont}$, such that the following holds. There exist $c>0, n_0\inn$, $(F_k,\si_k)\in \mathscr{G}_{n_0}$ for $k\inn$ and $\si\in\cantor$ satisfying the following.

\begin{itemize}

\item[(i)] $|F_k(x_k)| > c$ for all $k\inn$.

\item[(ii)] The $F_k$ are pairwise disjoint.

\item[(iii)] $\si\neq\si_k$ for all $k\inn$.

\item[(iv)] $\si\wedge\si_k \cneq \si\wedge\si_{k+1}$ for all $k\inn$.

\item[(v)] $|\si\wedge\si_k| < \varmin(F_k,\si_k)$ for all $k\inn$.

\end{itemize}

Then $\{x_k\}_k$ generates an $\ell_1^n$ spreading model for every $n\inn$.\label{lemelloneskipped}
\end{lem}

\begin{proof}
By changing the signs of the $x_k$, we may assume that $F_k(x_k) > c$ for all $k\inn$.

Arguing in a similar manner as in the proof of Proposition \ref{prpnlarge} one can inductively prove that for every $n\inn$ and $G\in\mathcal{S}_n$ the following hold.

\begin{itemize}

\item[(a)] $(\cup_{k\in G}F_k, \si)\in\mathscr{G}_{n_0+n}$

\item[(b)] $\varmin(\cup_{k\in G}F_k, \si) = |\si\wedge\si_{\min G}|$ and $\varmax(\cup_{k\in G}F_k, \si) = |\si\wedge\si_{\max G}|$

\end{itemize}

 Since $\{x_k\}_k$ is unconditional, it is enough find positive constants $c_n>0$, such that fixing $G\in\mathcal{S}_n$ and $\{\la_k\}_{k\in G}$ non negative reals, we have the following.
\begin{equation*}
\|\sum_{k\in G}\la_kx_k\| > c_n\sum_{k\in G}\la_k
\end{equation*}
Properties (a) and (b), yield that $F = \cup_{k\in G}F_k\in\mathcal{G}_{n_0 + n}$. This means the following.
\begin{equation*}
\|\sum_{k\in G}\la_kx_k\| \geq \|P_{n_0+n}(\sum_{k\in G}\la_kx_k)\|_{n_0+n} = \frac{2}{2^{n_0+n}}\|\sum_{k\in G}\la_kx_k\|_{n_0+n} > \frac{2c}{2^{n_0+n}}\sum_{k\in G}\la_k
\end{equation*}
\end{proof}

\begin{lem}
Let $\{x_k\}_k$ be a normalized, disjointly supported block sequence of $\{\tilde{e}_\xi\}_{\xi<\cont}$, such that the following holds. There exist $c>0, n_0\inn, \si\in\cantor$, a sequence $\{F_k\}_k$ in $\mathcal{G}_{n_0}$ satisfying the following.
\begin{itemize}

\item[(i)] $|F_k(x_k)| > c$ for all $k\inn$.

\item[(ii)] The sets $F_k$ are pairwise disjoint

\item[(iii)] $(F_k,\si)\in\mathscr{G}_{n_0}$ for all $k\inn$

\item[(iv)] $\varmax(F_k,\si) < \varmin(F_{k+1},\si)$, for all $k\inn$

\end{itemize}
Then $\{x_k\}_k$ generates an $\ell_1^n$ spreading model for every $n\inn$.\label{lemelloneattached}
\end{lem}

\begin{proof}

The proof is identical to the proof of Lemma \ref{lemelloneskipped}.

\end{proof}

\begin{lem}
Let $\{x_k\}_k$ be a sequence in $\X$ and $n\inn$ such that $\lim_k\|P_nx_k\|_n = 0$. Then for every $\e > 0$ there exists $k_0\inn$ such that for every $k\geqslant k_0$ the following holds.
\begin{equation*}
|F(x_k)| < \e\quad\text{for every}\;F\in\mathcal{G}_{n}
\end{equation*}\label{lempnzero}
\end{lem}

\begin{proof}

Fix $\e>0$. Choose $k_0\inn$, such that $\|P_{n}x_k\|_{n} = \frac{1}{2^{n}}\|x_k\|_{n} < \frac{1}{2^{n}}\e$, for every $k\geqslant k_0$. By definition of the norm $\|\cdot\|_n$, this means the following.
\begin{equation*}
|F(x_k)| < \e\quad\text{for every}\;F\in\mathcal{G}_{n}
\end{equation*}
\end{proof}

\begin{lem}
Let $\{x_k\}_k$ be a normalized, disjointly supported block sequence of $\{\tilde{e}_\xi\}_{\xi<\cont}$, such that $\lim_k\|x_k\|_\infty = 0$ and there exists $n\inn$ such that $\lim\sup_k\|P_nx_k\|_n > 0$. Assume moreover, that if $n_0 = \min\{n: \lim\sup_k\|P_nx_k\|_n > 0\}$, there exists $c>0, \si\in\cantor$ and $\{F_k\}_k$ a sequence in $\mathcal{G}_{n_0}$ satisfying the following.
\begin{itemize}

\item[(i)] $|F_k(x_k)| > c$ for all $k\inn$.

\item[(ii)] The sets $F_k$ are pairwise disjoint

\item[(iii)] $(F_k,\si)\in\mathscr{G}_{n_0}$ for all $k\inn$

\end{itemize}
Then $\{x_k\}_k$ has a subsequence that generates an $\ell_1^n$ spreading model for every $n\inn$.\label{lemFkattached}
\end{lem}

\begin{proof}

We shall prove that for every $k_0, m$ natural numbers, there exist $k\geqslant k_0$ and $G_k\subset F_k$ such that $|G_k(x_k)| > c/2$ and $\varmin(G_k,\si) > m$.

If the above statement is true, we may clearly choose $\{G_k\}_k$ in $\mathcal{G}_{n_0}$ satisfying the assumptions of Lemma \ref{lemelloneattached}, which will complete the proof.

We assume that $n_0\geqslant 2$, as the case $n_0 = 1$ uses similar arguments and the fact that $\lim_k\|x_k\|_\infty = 0$.

Fix $k_0,m\inn$. By Lemma \ref{lempnzero}, choose $k\geqslant k_0$, such that the following holds.
\begin{equation}
|F(x_k)| < \frac{c}{2m}\quad\text{for every}\;F\in\mathcal{G}_{n_0-1}\label{eq2lemFk}
\end{equation}

We distinguish two cases.
\begin{itemize}

\item[{\em Case 1:}] There is $(F_i^k,\si_i^k)_{i=1}^d$ a skipped branching of $\si$ in $\mathscr{G}_{n_0-1}$ with $F_k = \cup_{i=1}^dF_i^k$.

\item[{\em Case 2:}] There is $(F_i^k,\si)_{i=1}^d$ an attached branching of $\si$ in $\mathscr{G}_{n_0-1}$ with $F_k = \cup_{i=1}^dF_i^k$.

\end{itemize}

In either case, by Proposition \ref{prphereditary} we have that if we set $G_k = \cup_{i=m+1}^dF_i^k$, then $(G_k,\si)\in\mathscr{G}_{n_0}$. Moreover, \eqref{eq2lemFk} yields that $|G_k(x_k)| > c/2$.

All that remains, is to show that $\varmin(G_k,\si) > m$.

If we are in case 1, then $\varmin(G_k,\si) = |\si\wedge\si_{m+1}^k|$. By Definition \ref{defskipped} we have that $|\si\wedge\si_{i}^k| < |\si\wedge\si_{i+1}^k|$ for $i=1,\ldots,m$, which of course yields that $|\si\wedge\si_{m+1}^k| > m$.

If, on the other hand, we are in case 2, then $\varmin(G_k,\si) = \varmin(F_{m+1}^k,\si)$. By Definition \ref{defattached} we have that $\varmin(F_i^k,\si) \leqslant \varmax(F_{i}^k,\si) < \varmin(F_{i+1}^k,\si)$ for $i=1,\ldots,m$, which yields that $\varmin(F_{m+1}^k,\si) > m$.

\end{proof}

\begin{lem}
Let $\{x_k\}_k$ be a normalized, disjointly supported block sequence of $\{\tilde{e}_\xi\}_{\xi<\cont}$, such that there exists $n\inn$ such that $\lim\sup_k\|P_nx_k\|_n > 0$. Then, passing if necessary, to a subsequence, there exist $c>0$ and $(F_k,\si_k)\in\mathscr{G}_n$ satisfying the following.

\begin{itemize}

\item[(i)] The $F_k$ are pairwise disjoint.

\item[(ii)] $|F_k(x_k)| > c$ for all $k\inn$.

\end{itemize}\label{lempnnotzero}
\end{lem}

\begin{proof}
Pass to a subsequence of $\{x_k\}_k$ and choose $\e>0$, such that the following holds.
\begin{equation*}
\|P_nx_k\|_n = \frac{1}{2^n}\|x_k\|_n > \e,\quad \text{for all}\;k\inn.
\end{equation*}
 By the definition of the norm $\|\cdot\|_n$, there exist $(F_k,\si_k)\in\mathscr{G}_n$ with $|F_k(x_k)| > 2^n\e$, for all $k\inn$. By virtue of Proposition \ref{prphereditary} and the fact that $\{x_k\}_k$ is disjointly supported, we may assume that the $F_k$ are pairwise disjoint. Setting $c = 2^n\e$ finishes the proof.
\end{proof}

\begin{prp}
Let $\{x_k\}_k$ be a normalized, disjointly supported block sequence of $\{\tilde{e}_\xi\}_{\xi<\cont}$, such that $\lim_k\|x_k\|_\infty = 0$ and there exists $n\inn$ such that $\lim\sup_k\|P_nx_k\|_n > 0$. Then $\{x_k\}_k$ has a subsequence that generates an $\ell_1^n$ spreading model for every $n\inn$.\label{prppnnotzero}
\end{prp}

\begin{proof}
Set $n_0 = \min\{n: \lim\sup_k\|P_nx_k\|_n > 0\}$ and as in the proof of Lemma \ref{lemFkattached} let us assume that $n_0\geqslant 2$. Apply Lemmas \ref{lempnnotzero} and \ref{lempnzero}, pass to a subsequence of $\{x_k\}_k$ and find $c>0$, $(F_k,\si_k)\in\mathscr{G}_{n_0}$ such that the following are satisfied.
\begin{itemize}

\item[(i)] The $F_k$ are pairwise disjoint.

\item[(ii)] $|F_k(x_k)| > c$ for all $k\inn$.

\item[(iii)] $|F(x_k)| < c/4$ for every $k\inn$ and $F\in\mathcal{G}_{n_0-1}$.

\end{itemize}

Passing to a further subsequence, choose $\si\in\cantor$ such that $\lim_k\si_k = \si$. We distinguish two cases.

\begin{itemize}

\item[{\em Case 1:}] $\lim_k\max\{|G(x_k)|: G\subset F_k$\; with \;$(G,\si)\in\mathscr{G}_{n_0}\} = 0$

\item[{\em Case 2:}] $\lim\sup_k\max\{|G(x_k)|: G\subset F_k$\; with \;$(G,\si)\in\mathscr{G}_{n_0}\} > 0$

\end{itemize}

Let us first treat case 1. Pass once more to a subsequence of $\{x_k\}_k$, satisfying the following.

\begin{itemize}

\item[(a)] $\max\{|G(x_k)|: G\subset F_k$\; with \;$(G,\si)\in\mathscr{G}_{n_0}\} < c/4$, for all $k\inn$.

\item[(b)] $\si\neq\si_k$, for every $k\inn$.

\item[(c)] $\si\wedge\si_k \cneq \si\wedge\si_{k+1}$ for all $k\inn$.

\end{itemize}

We shall prove the following. For every $k$, there exists $G_k\subset F_k$, such that the following hold.
\begin{itemize}

\item[(d)] $|G_k(x_k)| > c/2$

\item[(e)] $|\si\wedge\si_k| < \varmin(G_k,\si_k)$

\end{itemize}

Combining (b), (c), (d) and (e), we conclude that the assumptions of Lemma \ref{lemelloneskipped} are satisfied, which proves the desired result, in case 1.

Set $G_k^{\prime\prime} = \{\tau\in F_k: \si_k\wedge\tau = \si\wedge\tau\}$. Proposition \ref{prphereditary} and Lemma \ref{lemsiprime} yield that $(G_k^{\prime\prime},\si)\in\mathscr{G}_{n_0}$. Setting $F_k^\prime = F_k\setminus G_k^{\prime\prime}$, property (a) yields that $|F_k^\prime(x_k)| > 3c/4$.

Set $G_k^\prime = \{\tau\in F_k^\prime: \si_k\wedge\tau \cneq \si\wedge\tau\}$. Once more, Proposition \ref{prphereditary} yields that $(G_k^\prime,\si_k)\in\mathscr{G}_{n_0}$, however Lemma \ref{lemnminusone} yields that $G_k^\prime\in\mathcal{G}_{n_0-1}$ and therefore, by (iii) we have that $|G_k^\prime(x_k)| < c/4$.

Set $G_k = F_k^\prime\setminus G_k^\prime$. Then we have that $|G_k(x_k)| > c/2$, i.e. (d) holds.

We will show that (e) also holds. By Corollary \ref{corminmax}, there exists $\tau\in G_k$, with $\varmin(G_k,\si_k) = |\si_k\wedge\tau|$. Since $\tau\notin G_k^{\prime\prime}$, we have that $|\si_k\wedge\tau| \neq |\si\wedge\tau|$.

We will show that $|\si\wedge\tau| < |\si_k\wedge\tau|$. Assume that this is not the case, i.e. $|\si_k\wedge\tau| < |\si\wedge\tau|$. In other words, $\si_k\wedge\tau \cneq \si\wedge\tau$. This means that $\tau\in G_k^\prime$, a contradiction.

We conclude that $\si\wedge\tau \cneq \si_k\wedge\tau$. Lemma \ref{littlelemma} yields that $\si\wedge\tau = \si_k\wedge\si$. Applying Lemma \ref{littlelemma} once more, we conclude that $\si\wedge\si_k \cneq \si_k\wedge\tau$, i.e. $|\si\wedge\si_k| < |\si_k\wedge\tau| = \varmin(G_k,\si_k)$, which completes the proof for case 1.

It only remains to treat case 2. Observe, that in this case, we may easily pass to a subsequence of $\{x_k\}_k$, satisfying the assumptions of Lemma \ref{lemFkattached}. This completes the proof.

\end{proof}

Combining Propositions \ref{supnorm}, \ref{pnzero} and \ref{prppnnotzero}, one obtains the following.

\begin{prp}
Let $\{x_k\}_k$ be a normalized weakly null sequence in $\X$. Then $\{x_k\}_k$ has a subsequence that generates an $\ell_1^n$ spreading model for every $n\inn$.\label{prpell1smwn}
\end{prp}

\begin{prp}
The space $\X$ is saturated with subspaces of Tsirelson space.\label{prpTsat}
\end{prp}

\begin{proof}
It is an immediate consequence of Proposition \ref{prpell1smwn} that $\X$ does not contain a copy of $c_0$. By Proposition \ref{prpc0sat}, the spaces $X_n$ are $c_0$ saturated and therefore, the operators $P_n:\X\rightarrow X_n$, are strictly singular.

We conclude, that in any infinite dimensional subspace $Y$ of $\X$, $n_0\inn$ and $\e>0$, there exists $x\in Y$ with $\|x\| = 1$ and $\|P_nx\|_n < \e$ for $n=1,\ldots,n_0$. One may easily construct a normalized sequence in $Y$, satisfying the assumption of Proposition \ref{pnzero}, which completes the proof.
\end{proof}

In particular, the previous result yields that neither $c_0$ nor $\ell_1$ embed into $\X$. Using James' well known theorem for spaces with an unconditional basis, we conclude the following.

\begin{cor}
The space $\X$ is reflexive.\label{correflexive}
\end{cor}

\begin{rmk}
As is well known, (see \cite{AKT}, Lemma 37), if $\{x_k\}_k$ is a normalized weakly null sequence in a Banach space $X$ and $x\in X$, then $\{x_k\}_k$ admits an $\ell_1$ spreading model, if and only if $\{x_k - x\}_k$ admits an $\ell_1$ spreading model as well. Since $\X$ is reflexive and every normalized weakly null sequence admits an $\ell_1$ spreading model, we conclude that any bounded sequence in $\X$, without a norm convergent subsequence, admits an $\ell_1$ spreading model. In other words, every spreading model admitted by $\X$, is either trivial or equivalent to the usual basis of $\ell_1$.
\end{rmk}

\section{Spaces admitting $\ell_1^\al$ as a unique spreading model}

The goal of the present section, is to give an outline of the construction, for a given countable ordinal number $\al$, of a non separable reflexive space $\mathfrak{X}_{\cont}^\al$, having the following property. Every normalized weakly null sequence in $\mathfrak{X}_{\cont}^\al$ has a subsequence that generates an $\ell_1^\al$ spreading model.

\begin{dfn}

Let $\al$ be a countable ordinal number. Define $\|\cdot\|_{T_\al}$ to be the unique norm on $c_{00}(\mathbb{N})$ that satisfies the following implicit formula, for every $x\in c_{00}(\mathbb{N})$.

\begin{equation*}
\|x\|_{T_\al} = \max\big\{\|x\|_\infty,\; \frac{1}{2}\sup\sum_{i=1}^d\|E_ix\|_{T_\al}\big\}
\end{equation*}
where the supremum is taken over all $E_1<\cdots<E_d$ subsets of the natural numbers with $\{\min E_i:\; i=1,\ldots,d\}\in\mathcal{S}_\al$.

Define the Tsirelson space of order $\al$, denoted by $T_\al$, to be the completion of $c_{00}(\mathbb{N})$ with the aforementioned norm.
\end{dfn}

The space $T_\al$ is reflexive and the unit vector basis $\{e_n\}_n$, forms a 1-unconditional basis for $T_\al$. Moreover, every normalized weakly null sequence in $T_\al$, has a subsequence that generates an $\ell_1^\al$ spreading model. For more details see \cite{AT}.

Given a countable ordinal number $\al$, we shall construct $\{\mathcal{G}_n^\al\}_n$ an increasing sequence of families of finite subsets of $\cantor$, strongly related to $\{\mathcal{G}_n\}_n$. As before, we first define some auxiliary families $\mathscr{G}_n^\al$, $n\inn$.

\begin{dfn}
We define $\mathscr{G}_1^\al$ to be all pairs $(F,\si)$, where $F = \{\tau_i\}_{i=1}^d\in[\cantor]^{<\omega}$, $d\inn$ and $\si\in\cantor$, such that the following are satisfied.

\begin{itemize}

\item[(i)] $\si \neq \tau_i$ for $i=1,\ldots,d$

\item[(ii)] $\si\wedge\tau_1\neq\varnothing$ and if $d>1$, then $\si\wedge\tau_1 \cneq \si\wedge\tau_2 \cneq \cdots \cneq \si\wedge\tau_d$

\item[(iii)] $\{|\si\wedge\tau_i|:\; i=1,\ldots,d\}\in\mathcal{S}_\al$

\end{itemize}
Define $\varmin(F,\si) = |\si\wedge\tau_1|$ and $\varmax(F,\si) = |\si\wedge\tau_d|$.\label{defg1al}
\end{dfn}

Assume that $n\inn$, $\mathscr{G}_{k}^\al$ have been defined for $k\leqslant n$ and that for $(F,\si)\in\mathscr{G}_k^\al$, $\varmin(F,\si)$ and $\varmax(F,\si)$ have also been defined.

\begin{dfn}
Let $(F_i,\si_i)_{i=1}^d$, $d\inn$ be a finite sequence of elements of $\mathscr{G}_{n}^\al$ and $\si\in\cantor$.

We say that $(F_i,\si_i)_{i=1}^d$ {\em is a skipped branching of $\si$ in $\mathscr{G}_n^\al$}, if the following are satisfied.

\begin{itemize}

\item[(i)] The $F_i, i=1,\ldots,d$ are pairwise disjoint

\item[(ii)] $\si\neq\si_i$ for $i=1,\ldots,d$

\item[(iii)] $\si\wedge\si_1\neq\varnothing$ and if $d>1$, then $\si\wedge\si_1 \cneq \si\wedge\si_2 \cneq \cdots \cneq \si\wedge\si_d$

\item[(iv)] $|\si\wedge\si_i| < \varmin(F_i,\si_i)$ for $i=1,\ldots,d$

\item[(v)] $\{|\si\wedge\si_i|:\;i=1,\ldots,d\}\in\mathcal{S}_\al$

\end{itemize}\label{defskippedal}
\end{dfn}

\begin{dfn}
Let  $\si\in\cantor$ and $(F_i,\si)_{i=1}^d$, $d\inn$ be a finite sequence of elements of $\mathscr{G}_{n}^\al$.

We say that $(F_i,\si)_{i=1}^d$ {\em is an attached branching of $\si$ in $\mathscr{G}_{n}^\al$} if the following are satisfied.

\begin{itemize}

\item[(i)] The $F_i, i=1,\ldots,d$ are pairwise disjoint

\item[(ii)] If $d>1$, then $\varmax(F_i,\si) < \varmin(F_{i+1},\si)$, for $i=1,\ldots,d-1$

\item[(iii)] $\{\varmin(F_i,\si):\;i=1,\ldots,d\}\in\mathcal{S}_\al$

\end{itemize}\label{defattachedal}

\end{dfn}

We are now ready to define $\mathscr{G}_{n+1}^\al$.

\begin{dfn}
We define $\mathscr{G}_{n+1}^\al$ to be all pairs $(F,\si)$, where $F\in[\cantor]^{<\omega}$ and $\si\in\cantor$, such that one of the following is satisfied.

\begin{itemize}

\item[(i)] $(F,\si)\in\mathscr{G}_{n}^\al$.

\item[(ii)] There is $(F_i,\si_i)_{i=1}^d$ a skipped branching of $\si$ in $\mathscr{G}_{n}^\al$ such that $F = \cup_{i=1}^d F_i$.

In this case we say that $(F,\si)$ {\em is skipped}. Moreover set $\varmin(F,\si) = |\si\wedge\si_1|$ and $\varmax(F,\si) = |\si\wedge\si_d|$.

\item[(iii)] There is $(F_i,\si)_{i=1}^d$ an attached branching of $\si$ in $\mathscr{G}_{n}^\al$ such that $F = \cup_{i=1}^d F_i$.

In this case we say that $(F,\si)$ {\em is attached}. Moreover set $\varmin(F,\si) = \varmin(F_1,\si)$ and $\varmax(F,\si) = \varmax(F_d,\si)$.

\end{itemize}
\label{defgnal}
\end{dfn}

\begin{dfn}
For a countable ordinal number $\al$ and $n\inn$ we define
\begin{equation*}
\mathcal{G}_n^\al = \{F\subset\cantor:\;\text{there exists}\;\si\in\cantor\;\text{with}\;(F,\si)\in\mathscr{G}_n^\al\}\cup\{\varnothing\}
\end{equation*}\label{defggnal}
\end{dfn}

\begin{rmk}
It is clear that for $\al = 1$, $\mathscr{G}_n^\al = \mathscr{G}_n$, for every $n\inn$. Moreover, for a countable ordinal number $\al$, every result stated for $\mathscr{G}_n$, up to Proposition \ref{prphereditary}, holds also for $\mathscr{G}_n^\al$ and the proofs are identical. On the other hand, if for $n\inn$ we denote by $\mathcal{S}_\al^n$ the convolution of $\mathcal{S}_\al$ with itself $n$ times, Proposition \ref{prpnlarge} can be restated as follows.\label{rmkidentical}
\end{rmk}

\begin{prp}
Let $\al$ be a countable ordinal number. Then  for every $B$ infinite subset of $\cantor$ there exists a one to one map $\phi:\mathbb{N}\rightarrow B$ with $\phi(F)\in\mathcal{G}_n^\al$ for every $F\in\mathcal{S}_\al^n$ and $n\inn$.\label{prpalnlarge}
\end{prp}

Theorem \ref{thmcompact} takes the following form and the proof uses the compactness of $\mathcal{S}_\al$ and Corollary \ref{corcomal}.

\begin{thm}
Let $\al$ be a countable ordinal number and $n\inn$. Then $\mathcal{G}_n^\al$ is an $\al$-large, hereditary and compact family of finite subsets of $\cantor$.\label{thmcompactal}
\end{thm}

In order to define the desired space $\X^\al$, one takes the same steps as in the previous section. All proofs are identical.

\begin{dfn}
For $\al$ a countable ordinal number and $n\inn$ define a norm on $c_{00}(\cont)$ in the following manner.
\begin{itemize}

\item[(i)]
For $n\inn$, we may identify an $F\in\mathcal{G}_n^\al$ with a linear functional $F:c_{00}(\cont)\rightarrow\mathbb{R}$ in the following manner. For $x = \sum_{\xi<\cont}\la_\xi e_\xi \in c_{00}(\cont)$
\begin{equation*}
F(x) = \sum_{\tau\in F}\la_{\xi_\tau}
\end{equation*}

\item[(ii)]
For $x\in c_{00}(\cont)$ define
\begin{equation*}
\|x\|_n^\al = \sup\{|F(x)|: F\in\mathcal{G}_n^\al\}
\end{equation*}
Set $X_n^\al$ to be the completion of $(c_{00}(\cont), \|\cdot\|_n^\al)$.
\end{itemize}\label{defxnal}

\end{dfn}

\begin{dfn}
Define the following norm on $c_{00}(\cont)$. For $x\in c_{00}(\cont)$
\begin{equation*}
\|x\| = \big\|\sum_{n=1}^\infty\frac{1}{2^n}\|x\|_n^\al e_n\big\|_{T_\al}
\end{equation*}
Set $\X^\al$ to be the completion of $(c_{00}(\cont), \|\cdot\|)$.\label{defxal}
\end{dfn}

\begin{thm}
The space $\X^\al$ is a non separable reflexive space with a suppression unconditional Schauder basis with the length of the continuum, having the following property. Every normalized weakly null sequence in $\X^\al$ has a subsequence that generates an $\ell_1^\al$ spreading model.\label{mainthmal}
\end{thm}

\end{document}